\documentclass{amsart}
\usepackage{amssymb, amscd}
\usepackage{enumerate}
\usepackage[dvips,dvipdf]{graphicx}
\usepackage[usenames,dvipsnames]{color}

\newcommand\vt[1]{ \mathbf{#1} }

\newcommand{\maC}{\mathcal C}
\newcommand{\maD}{\mathcal D}
\newcommand{\maE}{\mathcal E}

\newcommand{\maK}{\mathcal K}

\newcommand{\maS}{\mathcal S}

\newcommand{\ind}{\operatorname{ind}}

\newcommand{\CC}{\mathbb C}

\newcommand{\PP}{\mathbb P}
\newcommand{\RR}{\mathbb R}

\newcommand{\TT}{\mathbb T}
\newcommand{\ZZ}{\mathbb Z}

\newcommand\Hk{H_{\vt k}}

\newcommand\TmS{\TT \smallsetminus \maS}
\newcommand\RmS{\overline{(\RR^3 \smallsetminus \maS)}_{rad}}

\newtheorem{theorem}{Theorem}[section]
\newtheorem{proposition}[theorem]{Proposition}
\newtheorem{corollary}[theorem]{Corollary}

\newtheorem{lemma}[theorem]{Lemma}

\theoremstyle{definition}

\theoremstyle{remark}

\author[E. Hunsicker]{Eugenie Hunsicker} \address{Eugenie Hunsicker,
Department of Mathematical Sciences, Loughborough University, Loughborough,
Leicestershire, LE11 3TU, UK } \email{E.Hunsicker@lboro.ac.uk}

\author[H. Li]{Hengguang Li} \address{Hengguang Li, Department
of Mathematics, Wayne State University,  Detroit, MI 48202, USA} \email{hli@math.wayne.edu}

\author[V. Nistor]{Victor Nistor} \address{V. Nistor, Pennsylvania
  State University, Math. Dept., University Park, PA 16802, USA, and
  Inst. Math. Romanian Acad.  PO BOX 1-764, 014700 Bucharest Romania}
\email{nistor@math.psu.edu}

\author[V. Uski]{Ville Uski} \address{Ville Uski,
Department of Mathematical Sciences, Loughborough University, Loughborough,
Leicestershire, LE11 3TU, UK } \email{V.Uski@lboro.ac.uk}

\date\today \thanks{V.N. was partially supported by the NSF Grants
  OCI-0749202 and DMS-1016556. Manuscripts available from {\bf
    http:{\scriptsize//}www.math.psu.edu{\scriptsize/}nistor{\scriptsize/}}.
  H.L. was partially supported by the NSF Grant DMS-1115714.  E.H. was
  supported in part by Leverhulme Trust Project Assistance Grant F/00
  261/Z.}

\begin{document}

\title[Schr\"odinger operators]{Analysis of Schr\"odinger operators
  with inverse square potentials I: regularity results in 3D}

\date{\today}

%\subjclass{Primary: ???, Secondary: ????}

\begin{abstract}
Let $V$ be a potential on $\RR^3$ that is smooth everywhere except at
a discrete set $\maS$ of points, where it has singularities of the
form $Z/\rho^2$, with $\rho(x) = |x - p|$ for $x$ close to $p$ and $Z$
continuous on $\RR^3$ with $Z(p) > -1/4$ for $p \in \maS$. Also assume
that $\rho$ and $Z$ are smooth outside $\maS$ and $Z$ is smooth in
polar coordinates around each singular point. We either assume that
$V$ is periodic or that the set $\maS$ is finite and $V$ extends to a
smooth function on the radial compactification of $\RR^3$ that is
bounded outside a compact set containing $\maS$. In the periodic case,
we let $\Lambda$ be the periodicity lattice and define $\TT := \RR^3/
\Lambda$. We obtain regularity results in weighted Sobolev space for
the eigenfunctions of the Schr\"odinger-type operator $H = -\Delta +
V$ acting on $L^2(\TT)$, as well as for the induced $\vt
k$--Hamiltonians $\Hk$ obtained by restricting the action of $H$ to
Bloch waves. Under some additional assumptions, we extend these
regularity and solvability results to the non-periodic case. We sketch
some applications to approximation of eigenfunctions and eigenvalues
that will be studied in more detail in a second paper.
\end{abstract}

\maketitle

\tableofcontents

\section{Introduction and statement of main results}

We study in this paper regularity and decay properties of the
eigenfunctions of Schr\"{o}dinger type operators with inverse-square
singularities. We either assume that the potential is periodic or that
it has a nice behavior at infinity and only finitely many
singularities. In order to explain our assumptions and results in more
detail, we organize our Introduction in subsections, concentrating on
the case of periodic potentials, the non-periodic case being similar,
but simpler. We first introduce the operators $\Hk$ obtained from the
Hamiltonian $-\Delta + V$ acting on Bloch waves. In the second
subsection of the Introduction, we explain our assumptions on the
potential $V$. Finally, we state our main results and we summarize the
contents of the paper.

This paper is written to put on a solid foundations the numerical
methods developed in \cite{HLNU2} and \cite{HLNU4}. We have thus
written this paper with an eye to the numerical analyst. More
theoretical results on Hamiltonians with inverse square potentials in
arbitrary dimensions will be included in the third part of this paper.

We have to mention Kato's ground breaking papers \cite{Kato51}, where
the self-adjointness of Schr\"{o}dinger type Hamiltonians was proved
and \cite{Kato57}, where boundedness properties of the eigenfunctions
and eigenvalues of these Hamiltonian operators was proved. Moreover,
see \cite{Bespalov, Westphal, Felli2, Felli1, MorozSchmidt, Zuazua,
  Sprung} for other papers studying Hamiltonians with inverse square
potentials, both from the point of view of physical and numerical
applications. See also \cite{Dauge, ferrero, Flad2, Flad3, Flad1,
  Fournais1, Hamaekers, Siedentop, Schwab, VasyReg, Yserentant3} for
some related results.

\subsection{The Hamiltonian $\Hk$}
Let $V$ be a periodic potential on $\RR^3$ with Bravais lattice (of
translational symmetries) $\Lambda \cong \mathbb{Z}^3$. Assume that
$V$ is smooth except at a set of points $\maS$, which is thus
necessarily also periodic with with respect to the lattice
$\Lambda$. We assume that there are only finitely many elements of
$\maS$ in any fundamental domain $\PP$ of $\Lambda$. Let $p \in \maS$
be a singular point and $\rho(x) = |x-p|$ for $x$ close to $p$ and
$\rho$ smooth outside $\maS$. We assume that around $p$ the potential
$V$ has a singularity of the form $Z/\rho^2$, where $Z$ is continuous
across $p$ and smooth in polar coordinates around $p$. We shall study
numerically Hamiltonian operators of the form
\begin{equation}\label{eq.def.H}
  H := -\Delta + V.
\end{equation}
Systems with such potentials have been studied as theoretical models
both from the viewpoint of classical mechanics and from the quantum
mechanical viewpoint.  In addition, they arise in a variety of
physical contexts, such as in relativistic quantum mechanics from the
square of the Dirac operator coupled with an interaction potential, or
from the interaction of a polar molecule with an electron
\cite{MorozSchmidt}.

A standard method for studying these operators is through their action
on Bloch waves.  For any $\vt k \in \RR^3$, the Bloch waves of $H$
with wave vector $\vt k$ are elements of $L^2_{loc}(\RR^3)$ that
satisfy the semi-periodicity condition that, for all $X \in \Lambda$,
\begin{equation}\label{eq.Block}
  \psi_{\vt k}(x+X) = e^{i{\vt k}\cdot X} \psi_{\vt k}(x).
\end{equation}
(It is enough to consider $\vt k$ in the first Brillouin zone $\PP_*$
of the reciprocal lattice to $\Lambda$. Also, the equality is that of
two $L^2_{loc}$ functions, and hence it holds only almost everywhere
in $x$.)  A Bloch wave with wavevector $\vt k$ can be written as
\begin{equation}\label{eq.Block2}
  \psi_{\vt k}(x) = e^{i{\vt k}\cdot x} u_{\vt k} (x)
\end{equation}
for a function $u_{\vt k}$ that is truly periodic with respect to
$\Lambda_t$ and thus can be considered as living on the three-torus
$\TT := \RR^3 /\Lambda \simeq (S^1)^3$ (obtained by identifying points
in $\RR^3$ that are equivalent under the action of the lattice
$\Lambda$ by translations). Note that the periodicity condition that a
Bloch wave satisfies prevents it from being in $L^2(\RR^3)$, thus a
nontrivial Bloch wave that satisfies the equation
\begin{equation}\label{eq.Bloch.e}
  H \psi_{\vt k} = \lambda \psi_{\vt k}
\end{equation}
is not, properly speaking, an eigenfunction of the Hamiltonian
operator $H$. Rather, it is a generalized eigenfunction, corresponding
to a value in the continuous spectrum of $H$. If $\psi_{\vt k}$ is a
Bloch wave that is a generalized eigenfunction of $H$ with generalized
eigenvalue $\lambda$, then the function $u_{\vt k} := e^{ -i{\vt
    k}\cdot x} \psi_{\vt k}(x)$ will then be an actual
$\lambda$-eigenfunction of the $\vt k$--Hamiltonian $\Hk$ on
$L^2(\TT)$ defined by
\begin{equation}\label{eq:hamk}
  \Hk := - \sum_{j=1}^3 (\partial_j + ik_j)^2 + V.
\end{equation}
Indeed, this follows from the equation
\begin{equation}\label{eq.H-Hk}
  H (e^{i{\vt k}\cdot x} u_{\vt k} (x)) = e^{i{\vt k}\cdot x} \Hk
  u_{\vt k} (x).
\end{equation}
Thus, it is useful to understand the regularity of eigenfunctions
$u_{\vt k}$ for the operators $\Hk$, as well as to arrive at
theoretical estimates for the accuracy of various schemes to estimate
them and their associated eigenvalues.

\subsection{Assumptions on the potential $V$}
In this paper, we extend and test the results of
\cite{HunsickerNistorSofo} to deal with the more singular potentials
that have inverse-square singularities. More precisely, we extend the
results of the aforementioned paper from potentials where $\rho V$ is
smooth in polar coordinates to potentials where $\rho^2V$ is smooth in
polar coordinates and continuous on $\TT$.  (Recall that $\rho$ is a
function that locally gives the distance to the singular point.) In
particular, we obtain regularity results in weighted Sobolev spaces
that will then permit us to derive estimates for the accuracy of two
approximation schemes that we design and which are studied in detail
in the forthcoming second and fourth parts of this paper.  The first
scheme is a finite element method with a mesh graded towards the
singular points as in \cite{HunsickerNistorSofo, BNZ3D2}.  The second
scheme is an augmented plane-wave method, similar to a ``muffin-tin''
method \cite{Martin}.

In order to state our results, we first need to set some notation and
introduce our assumptions on the potential $V$. Let $\maS \subset \TT
:= \RR^3/\Lambda$ be the finite set of points where $V$ has
singularities. By abuse of notation, we shall denote by $|x-y|$ the
induced distance between two points $x, y \in \TT$. Let then $\rho :
\TT \to [0, 1]$ be a nonnegative continuous function smooth outside
$\maS$ such that
\begin{equation}\label{eq.rho}
  \rho(x) = |x - p|\ \text{ for } x \text{ close to } p \in \maS,
\end{equation}
as before, and further assume also that $\rho(x) = 1$ for $x$ far from
$\maS$.

Our first assumption on $V$ is that $\rho^2V$ be smooth in polar
coordinates up to $\rho=0$ near each singularity. Let us explain this
in more detail. We first replace each singular point $p \in \maS$ with
a 2-sphere in a smooth way, thus obtaining a space denoted
$\overline{\TmS}$. This is the usual procedure of blowing up the
singularities. We think of stretching out the holes where the
singularities of $V$ are and compactifying the result using boundary
spheres.  It would be possible to carry out analysis similar to the
calculations in this paper with only the assumption that $\rho^2 V$ be
smooth on $\overline{\TmS}$. To simplify some calculations and to
obtain closed form results, however, we will further require the
resulting function $Z$ to be constant on the blow up of each point in
$\maS$, which can be reformulated as saying that $\rho^2 V$ is also
{\em continuous} on $\TT$. Our first assumption on the potential $V$
is therefore
\begin{equation}\label{eq.def.Z}
  \text{\bf Assumption 1}: \qquad\ Z := \rho^2 V \in
  \maC^\infty(\overline{\TmS}) \cap \maC(\TT).
\end{equation}
Assumption 1, more precisely the continuity of $Z$ at $\maS$, allows
us to formulate our second assumption. Namely,
\begin{equation}\label{eq.Z1/4}
  \text{\bf Assumption 2}: \qquad\
  \eta_0 := \min_{p \in \maS} Z(p) > -1/4 .
\end{equation}
Therefore, the constant
\begin{equation}\label{eq.eta}
  \eta := \sqrt{1/4 + \eta_0},
\end{equation}
which will play an important role in this paper, is a {\em positive
  real number}. This constant will appear in many results below. We
will use Assumptions 1 and 2 throughout the paper, except in Section
\ref{sec1}, where we prove more general forms of our results, not
requiring Assumption 2.

\subsection{Regularity and approximation results}
The domains of all the Hamiltonian operators considered in this paper
will be contained in weighted Sobolev spaces on $\TT \smallsetminus
\maS$. We define these spaces by:
\begin{equation}\label{eq.def.ws}
  \maK^m_a(\TmS ) := \{v : \TT \smallsetminus
  \maS \to \CC, \ \rho^{|\beta|-a} \partial^\beta v \in L^2(\TT),
  \ \forall\ |\beta| \leq m\}.
\end{equation}
These spaces have been considered in many other papers, most notably
in Kondratiev's groundbreaking paper \cite{kondratiev67}. They can be
identified with the b-Sobolev spaces of \cite{meaps} (associated to a
manifold with boundary), but with a different indexing and
notation. These spaces were generalized in \cite{AIN} to more general
manifolds with corners with additional structure (Lie manifolds).

To formulate the stronger regularity for eigenvalues, we shall need
the following notation. For each point $p \in \maS$, let
\begin{equation}\label{eq.nu_0}
\nu_0(p) = \left\{
\begin{array}{ll}
2 & Z(p) \geq \frac{3}{4}\\
1 + \sqrt{1/4 + Z(p)} & Z(p) \in
(-\frac{1}{4},\frac{3}{4})\\
1 & Z(p) \leq -\frac{1}{4}\,,
 \end{array}
 \right.
\end{equation}
 and
\begin{equation}\label{eq.nu0}
  \nu_0 = \min_{p \in  \maS} \nu_0(p).
\end{equation}
For each point $p \in \maS$ for which $Z(p) \in (-1/4,3/4]$, define a
smooth cutoff function $\chi_p$ that is equal to $1$ in a small
neighborhood of $p$ and is zero outside another small neighborhood of
$p$, so that all the functions $\chi_p$ have disjoint supports. Define
the space $W_{s}$ to be the complex linear span:
\begin{equation}\label{eq.def.Ws}
  	W_{s} = \sum_{Z(p) \in (-1/4,3/4]} \CC \chi_p \rho^{\sqrt {1/4
              + Z(p)} - 1/2}.
\end{equation}
Using also the notation introduced in the previous subsection, we then
have the following result, whose proof follows from the proof of
Theorem \ref{theorem1.gen} below.

\begin{theorem}\label{theorem1}
Consider a potential $V$ satisfying Assumptions 1 and 2. Then the
Hamiltonian operator $H_{\bf k}$ acting as an unbounded operator on
$L^2(\TT)$ has a distinguished self-adjoint extension with domain
\begin{equation*}
  \maD(H_{\bf k}) = \maK_{2}^{2}(\TmS) + W_{s} \subset
  \maK^2_{\nu}(\TmS), \quad \nu< \nu_0 = \min_{p \in \maS} \nu_0(p)
  \in (0,2].
\end{equation*}
In particular, if $\eta_0 :=\eta_0 := \min_{p \in \maS} Z(p) \geq
3/4$, then $\Hk$ is in fact essentially self-adjoint and, if $\eta_0 >
3/4$, then $\maD(H_{\bf k}) = \maK_2^2(\TmS)$.
\end{theorem}

The importance of the above theorem is the following corollary, which
says that under Assumptions 1 and 2, the Hamiltonian operators $\Hk =
-\sum_{j=1}^3 (\partial_j +ik_j)^2 +V$ acting on $L^2(\TT)$ can be
completely understood through their eigenfunctions and eigenvalues.

\begin{corollary} \label{corollary2}
Under the assumptions of Theorem \ref{theorem1}, there exists a
complete orthonormal basis of $L^2(\TT)$ consisting of eigenfunctions
of $\Hk$.
\end{corollary}

We can now state a regularity theorem for the eigenfunctions of $\Hk$
near a point $p \in \maS$, or equivalently, for Bloch waves associated
to the wavevector $\vt k$. Recall the functions $\chi_p$ supported
near points of $\maS$ and used to define the spaces $W_{s}$.

\begin{theorem}\label{theorem2}
Assume that $V$ satisfies Assumptions 1 and 2.  Let $\Hk u = \lambda
u$, where $u \in \maD(\Hk)$, $u \neq 0$. Then, for any $m \in \ZZ_+$,
\begin{equation*}
    u \in \maK^{m+1}_{a + 1}(\TmS), \quad \forall a < \eta := \min_{p
      \in \maS} \sqrt{1/4 + Z(p)}.
\end{equation*}
Moreover, we can find constants $a_p \in \RR$ such that
\begin{equation*}
    u - \sum_{p \in \maS} \chi_p \rho^{\sqrt{1/4 + Z(p)} - 1/2} \in
    \maK^{m+1}_{a' +1}(\TmS ), \quad \forall a' < \min_{p \in \maS}
    \sqrt{9/4 + Z(p)}\,.
\end{equation*}

\end{theorem}

The next result, which is the last we will mention in this
introduction, will permit us to construct approximation schemes for
the solutions of equations of the form $(\lambda + \Hk)u=f$.

\begin{theorem}\label{theorem1.5}
Let us use the notation of Theorem \ref{theorem1} and both Assumptions
1 and 2. Then there exists $C_0 > 0$ such that $\lambda + \Hk :
\maK_{a+1}^{m+1}(\TmS) \to \maK_{a-1}^{m-1}(\TmS)$ is an isomorphism
for all $m \in \ZZ_{\ge 0}$, all $|a| < \eta$, and all $\lambda >
C_0$. In particular, $H_k$ is symmetric and bounded below, thus has a
Friedrichs extension, which is equal to the closed extension
considered in Theorem \ref{theorem1} above.
\end{theorem}

From now on, we shall write $\Hk$ for the Friedrichs extension of the
original operator defined in Equation \eqref{eq.H-Hk} and $\maD(\Hk)$
for its domain.

We observe additionally that with the exception of Corollary
\ref{corollary2}, all of the results above extend to Hamiltonian
operators on $\RR^3$ associated to a non-periodic potential with a
finite number of inverse square singularities satisfying Assumptions 1
and 2 and radial limits at infinity. This is because the techniques
employed to obtain the results are local, and a Hamiltonian operator
over $\RR^3$ with a smooth potential that has radial limits at
infinity is always essentially self-adjoint. Of course, in this
situation, there are only isolated eigenvalues below the continuous
spectrum, and the bulk of the spectrum is continuous.

The paper is organized as follows.  In Section \ref{sec1} we prove
Theorem \ref{theorem1} identifying a closed self-adjoint extension of
the operator $H_k$, and Theorem \ref{theorem2} giving regularity
results for eigenfunctions of this closure. Some of the results of
this section do not rely on Assumption 2. Beginning with Section
\ref{sec2}, however, we shall require that Assumption 2 be
satisfied. In that section, we prove that for $\eta>-1/4$, $H_k$ is
bounded from below, and we can thus identify the closure from Section
\ref{sec1} as the Friedrichs extension of $\Hk$. In Section
\ref{sec.last} we discuss how our results extend to the the
nonperiodic case and how to use them in numerical methods.

\subsection*{Acknowledgements}
We would like to thank Alexander Strohmaier, Joerg Seiler, Thormas
Krainer, Jorge Sofo, and Anna Mazzucato for useful discussions. We
also thank the Leverhulme Trust whose funding supported the fourth
author during this project. This project was started while Hunsicker
and Nistor were visiting the Max Planck Institute for Mathematics in
Bonn, Germany, and we are grateful for its support.

\section{Regularity and singular values\label{sec1}}

The regularity analysis of the operators $\Hk$ is done locally in the
neighborhood of each $p \in \maS$. Let us recall that $Z := \rho^2 V
\in \maC^\infty(\overline{\TmS}) \cap \maC(\TT)$, which is our
Assumption 1, which we will require to hold true throughout this
paper. In this section, we shall mention explicitly when Assumption 2
is used, since some of the results hold in greater generality.

For simplicity of the notation, we shall assume that $\maS$ consists
of a single point $p$. The results for potentials with several
singularities with different values of $Z(p)$ can then be pieced
together from local versions of the result in the one singularity
case.  The proofs of Theorems \ref{theorem1} and Theorem
\ref{theorem2} (which we prove in this section in more general forms
not requiring Assumption 2), rely on the pseudodifferential operator
techniques of the b-calculus and b-operators \cite{aln2, grieser,
  meaps, Lesch, Schulze98}. A review of these basic tools is contained
in \cite{HunsickerNistorSofo}, so we will not go into detail about
them again here.  Throughout this paper, we will refer to b-operators
and the b-calculus, although the properties can be equivalently
described in terms of cone operators and the cone calculus, and in
fact, in some of the references in this section, they are referred to
in this way.  For a discussion of the equivalence of the b- and cone
calculi, see \cite{Joerg}.

\subsection{The boundary spectral set}
In order to use the b-calculus, we study the associated b-differential
operators
\begin{equation*}
  P_{{\bf k},\lambda}:= -\rho^2(\Hk - \lambda).
\end{equation*}
We can write such an operator in polar coordinates around $p \in \maS$
as
\begin{equation}\label{Hkinpolar}
  P_{{\bf k},\lambda} = (\rho \partial_\rho)^2 + \rho \partial_\rho +
  \Delta_{S^{n-1}} - \rho^2 V - \rho B_{{\bf k},\lambda},
\end{equation}
where
\begin{equation*}
  B_{{\bf k},\lambda}:= \rho\left(\sum_{j=1}^n (-2i\partial_j+k_j^2) -
  \lambda \right)
\end{equation*}
is a first order b-operator.

The operator $P_{\vt{k},\lambda}$ is an elliptic b-operator on
$\overline{\TmS}$.  We calculate the indicial family of
$P_{\vt{k},\lambda}$ at a point $p \in \maS$, denoted $
(\widehat{P_{\vt{k},\lambda}})_p(\tau)$, by replacing $\rho
\partial_\rho$ with $\tau$ in Equation \eqref{Hkinpolar}, and by
replacing the coefficients with their values at $\rho=0$. Doing this,
we find that the indicial families for $P_{\vt{k},\lambda}$ at $p \in
\maS$ are, in fact, independent of $\lambda$ and $\vt{k}$, since the
dependence on $\vt k$ and $\lambda$ affects only $B_{\vt k, \lambda}$,
and $\rho$ vanishes at the singular points. We summarize this
discussion in the following lemma.

\begin{lemma}\label{lemma.ind.fam}
Let $P:=P_{0,0}$.  Then the indicial family of $P_{\vt{k}, \lambda}$
at $p \in \maS$ is given by
\begin{equation}\label{Pindicial}
  (\widehat{P_{\vt{k},\lambda}})_p(\tau)= \hat{P}_p(\tau)= \tau^2 +
  \tau + \Delta_{S^2} - Z(p).
\end{equation}
\end{lemma}

Lemma \ref{lemma.ind.fam} allows us to calculate the boundary spectral
set $\mbox{Spec}_b(P_{p})$ for $P_{\vt{k}, \lambda}$ at a given
$p$. The {\em boundary spectral set} for $P_{\vt k}$ is then defined
by (\cite{meaps})
\begin{equation*}
  \mbox{Spec}_b(P_{p}):= \{(\tau,n) \mid \hat{P}_{p}(\tau)^{-1} \mbox{
    has a pole of order } n+1 \mbox{ at } \tau\}.
\end{equation*}
By Lemma \ref{lemma.ind.fam}, the set $\mbox{Spec}_b(P_{p})$ will also
be independent of $\vt{k}$ and $\lambda$. To calculate the spectral
set $\mbox{Spec}_b(P_p)$ explicitly, recall that the eigenvalues of
$\Delta_{S^2}$ are $-l(l+1)$, for $l \in \ZZ_{\geq 0}$, and define
\begin{eqnarray}
  \beta_{l,p}&:= &\frac{\sqrt{(1+2l)^2 + 4Z(p)} - 1}{2}\,, \quad l\in
  \ZZ_{\geq 0 } \\
  \alpha_{l,p}&:=& \frac{-\sqrt{(1+2l)^2 + 4Z(p)} - 1}{2}\,, \quad l
  \in \ZZ_{\geq 0} .
\end{eqnarray}
By an abuse of notation, we take $\sqrt{(1+2l)^2 + 4Z(p)}$ to denote
the positive imaginary root when the quantity under the root is
negative.  Our discussion gives the following.

\begin{lemma}\label{lemma.Spec_b}
If $Z(p) \notin \{ -(1/2+ l)^2\}_{l=0}^\infty$, we have
\begin{equation}\label{specPk}
  \mbox{Spec}_b(P_{p}) = \bigcup_{l \in \ZZ_{\geq 0}} \left\{
  \left(\beta_{l,p} ,0 \right) , \left(\alpha_{l,p} ,0 \right)
  \right\},
\end{equation}
and, if $Z(p) = -(1/2 + l_p)^2$, for some $l_p\geq 0$, $l_p \in \ZZ$,
then
\begin{equation}\label{specPk2}
  \mbox{Spec}_b(P_{p}) = \{(-1/2,1)\} \cup \bigcup_{l \in
    \ZZ_{\geq 0}} \left\{ \left(\beta_{l,p} ,0 \right) ,
  \left(\alpha_{l,p} ,0 \right) \right\}.
\end{equation}
\end{lemma}

When $Z(p) < -1/4$, a finite number of these values of $\alpha$ and
$\beta$ will be complex with real part equal to $-1/2$. This is one of
the reasons why we have introduced Assumption 2, which states that
$Z(p) > -1/4$ for all $p \in \maS$. Of course, if Assumption 2 is
satisfied, Spec${}_b(P_p)$ is given only by Equation
\eqref{specPk}. The case when $Z(p)$ is close to $-1/4$ is important
because it gives to some interesting numerical phenomena and also
because in various applications, it has an interesting interpretation,
see for instance \cite{MorozSchmidt}.

The machinery of the b-calculus now gives us information about closed
self-adjoint extensions of $\Hk$ (see \cite{Lesch, GM, GKM, SchroheSI}
for details).  Note that in this setting, we are considering
extensions of our operator acting on the core consisting of smooth
functions supported away from the points of $\maS$.  For $Z(p) < 3/4$,
there will be several possible self-adjoint extensions. Compare this
to the case of extending the Laplacian operator on $\TT$ from acting
on the core consisting of all smooth functions.  In this case, there
is a unique self-adjoint extension.  This is because the core is
larger than in our case.  The extension obtained using the larger core
is one of the possible extensions obtained using the smaller core, but
it is not the only one.  This is why one does not see the issue of
choosing a self-adjoint extension arising when the potential is of the
form $Z\rho^\alpha$ for $\alpha>-2$, for instance, in the Coulomb case
considered in \cite{HunsickerNistorSofo}.

With the above lemmas in place, we can now prove Theorem
\ref{theorem1}. In fact, we shall prove a stronger result that does
not require Assumption 2.

\begin{theorem}\label{theorem1.gen}
Consider a potential $V$ satisfying Assumption 1 and assume that
$\maS$ consists of just one point $p$. Then the Hamiltonian operator
$H_{\bf k}$ acting as an unbounded operator on $L^2(\TT)$ has
distinguished self-adjoint extension with domain $\maD(H_{\bf k})
\subset \maK^2_{\nu}(\TT \smallsetminus \maS )$ for all $\nu< \nu_0
\in (0,2]$.  In particular, if $Z(p)\geq 3/4$, then $\Hk$ is in fact
  essentially self-adjoint and, if $Z(p) > 3/4$, then $\maD(H_{\bf k})
  = \maK_2^2(\TmS)$. If Assumption 2 is satisfied, we also have
\begin{equation*}
\maD(H_{\bf k}) = \maK_2^2(\TmS) + \CC \chi \rho^{\eta -
      1/2}, \quad \eta = \sqrt{1/4 + Z(p)}.
\end{equation*}
where $\chi$ is a cutoff function that is zero outside some
neighborhood of $p$ and equals 1 close to $p$.
\end{theorem}

\begin{proof}
For each ${\vt k}$ and $\lambda$, the operator $\Hk - \lambda$ is a
symmetric, unbounded b-operator on $L^2(\TT)$ (see \cite{meaps, Lesch,
  GM}).  Define the operator $A= \rho^{1/2} \Hk \rho^{-1/2}$. Then $A$
is a symmetric unbounded b-operator on
$\rho^{-1}L^2_b(\TmS)=\mathcal{K}^0_{1/2}(\TmS)$. The self-adjoint
extensions of $A$ correspond exactly to those of $\Hk - \lambda$ with
domains shifted by weight $\rho^{1/2}$, so we will study the
self-adjoint extensions of $A$ as the calculations are somewhat easier
in this case.

By Lemma \ref{lemma.ind.fam}, the indicial roots of $A$ are the roots
of $\Hk$ shifted by 1/2. So we let
\begin{equation}
\label{eq:alpha}
  \tilde \beta_l = \sqrt{(l + 1/2)^2 +Z(p)} \quad \text{and} \quad
  \tilde \alpha_l = - \tilde \beta_l.
\end{equation}
Note that $0 \not= \tilde \beta_l\in \RR$ if
$Z(p)>-(l+\frac{1}{2})^2$, there is a double root at $\tilde
\beta_l=0$ if $Z(p)=-(l+\frac{1}{2})^2$, and $0 \neq \tilde \beta_l
\in i\RR$ if $Z(p)<-(l+\frac{1}{2})^2$.  The critical strip for
self-adjointness of unbounded operators on $\rho^{-1}L^2_b(\TmS)$ is
$(-1,1)$, that is, an operator is essentially self-adjoint if and only
if it has no indicial roots with real part in this interval (see
\cite{Lesch, GM}). Recalling that $l \in \ZZ_{\geq0}$, we see that for
$Z(p) \geq \frac{3}{4}$, there are no roots in the critical strip so
the operator is essentially self-adjoint.  If further $Z(p) >
\frac{3}{4}$ we get the somewhat stronger result that $\maD(H_{\bf k})
= \maK^2_{2}(\TmS )$.  For $Z(p) \in
\left(-\frac{1}{4},\frac{3}{4}\right)$ we get two real roots in the
critical strip corresponding to $l=0$, for $Z(p)= -\frac{1}{4}$, we
get a double root at 0 in the critical strip corresponding to $l=0$,
and for $Z(p)< -\frac{1}{4}$, we get a finite number of complex
conjugate imaginary root pairs and possibly two real roots or a double
root at 0 in the critical strip corresponding to some finite set of
$l$.

By the theory in \cite{GM}, the space $\maE:= {\rm Dom}(A_{\rm
  max})/{\rm Dom}(A_{\rm min})$ is finite dimensional and spanned by
functions local around $p$ of the form:
\begin{equation}\label{eq:defset}
  \bigcup_{|\Re(\tilde \beta_l)| \in (0,1)} \bigcup_{m=-l}^l \{
  w\rho^{\tilde \beta_l}\psi^m_l, w\rho^{-\tilde \beta_l}\psi^m_l \}
  \cup\bigcup_{\tilde \beta_l=0} \bigcup_{m=-l}^l \{w\psi^m_l, w\ln
  \rho \psi^m_l\},
\end{equation}
where $w$ is a local cutoff function that equals $1$ near $p$ and $0$
for $\rho$ large, and where the $\psi^m_l$ are an orthonormal basis
for spherical harmonics with eigenvalue $l(l+1)$.  Further, the
operator $A$ with domain $\maD := {\rm Dom}(A_{\rm min}) +
\operatorname{span}(u_1, \ldots, u_n)$ is self adjoint if, and only
if, linear combinations of these basis functions form a maximal set on
which the pairing, $[u,v]_A$ is trivial, where
\begin{equation}\label{Apairing}
  [u,v]_A := \frac{1}{2\pi} \oint_\gamma \hat{A} \hat{u}(\sigma)
  \cdot_{S^2} \overline{\hat{v}( \overline{\sigma})} \,d\sigma.
\end{equation}
Here $\gamma$ is a simple closed loop around the indicial roots of $A$
in the critical strip, $\hat{}$ represents the Mellin transform and
$\cdot_{S^2}$ denotes the standard $L^2$ paring on $S^2$.  Since the
$\psi^m_l$ are orthonormal, this pairing reduces to a sum of loop
integrals of the form:
\begin{equation}\label{pair}
  [u_l,v_l]_A = -\frac{1}{2\pi} \oint_\gamma (\sigma^2 + \tilde
  \beta_l^2) \hat{u}_l(\sigma)\overline{\hat{v}_l({\overline{\sigma}
    })} \,d\sigma,
\end{equation}
where $u_l = u_+ w \rho^{\tilde \beta_l} + u_- w \rho^{-\tilde
  \beta_l}$ and $v_l = v_+ w \rho^{\tilde \beta_l} + v_-w
\rho^{-\tilde \beta_l}$ if $\tilde \beta_l \neq 0$ and $u_l = u_+ w +
u_- w \log(\rho)$ and $v_l = v_+ w + v_-w \log(\rho)$ if $\tilde
\beta_l = 0$.

We can consider three cases: $\tilde \beta_l >0$, $\tilde \beta_l =0$
and $\tilde \beta_l \in i\RR$.  First, as in \cite{GM}, define
\begin{equation*}
  \Phi(\sigma) = \widehat{-\rho\partial_\rho w}(\sigma) := - \int_0^\infty
  \rho^{-i\sigma}w'(\rho) d\rho.
\end{equation*}
Then we get: $\overline{\hat{\Phi}(\bar{\sigma})} =
\hat{\Phi}(-\sigma)$ and $\Phi(0)=1$.  Also, using the properties of
the Mellin transform, we find that for any $\sigma \in \CC$,
\begin{equation*}
  \widehat{w\rho^{\pm \tilde \beta_l}}(\sigma) = \frac{\Phi(\sigma \pm
    i\tilde \beta_l)}{\sigma \pm i \tilde \beta_l}.
\end{equation*}

Now consider the case $\tilde \beta_l>0$.  Carrying out the loop
integral by evaluating residues, we arrive at the equation
\begin{equation*}
  [u_l,v_l]_A = k(u_+\bar{v}_- - u_-\bar{v}_+)
\end{equation*}
for a constant $k \neq 0$.  If we set $[u_l,u_l]=0$, this reduces to
$\arg(u_+)=\arg(u_-)$.  Thus to get a self-adjoint boundary condition
at $p$ for $A$ we can fix any ratio of $|u_+|$ to $|u_-|$.  We will
choose to enlarge the minimal domain by the set with $|u_-|=0$, so
spanned by $\{w \rho^{\tilde \beta_l} \psi^m_l\}_{m=-l}^l$.  Note that
if $Z(p)>-\frac{1}{4}$, we have $l=0$, so to create a self-adjoint
extension of $A = \rho^{1/2}\Hk \rho^{-1/2}$, we only need to expand
the minimal domain by the span of $w \rho^{\eta} = w \rho^{\sqrt{1/4
    +Z(p)}}$.

Next consider the case when $\tilde \beta_l =0$.  In this case we get
\begin{equation*}
  \hat{w}(\sigma) = \frac{\Phi(\sigma)}{\sigma} \quad \mbox{and} \quad
  \widehat{w\log \rho} = \frac{\Phi(\sigma)}{\sigma^2} -
  \frac{\Phi'(\sigma)}{\sigma}.
\end{equation*}
Carrying out the loop integral in this case, we arrive at the equation
\begin{equation*}
  [u_l,v_l]_A = k(u_+\bar{v}_- + u_-\bar{v}_+).
\end{equation*}
Again setting $[u_l,u_l]_A=0$, we arrive this time at the condition
$u_+\bar{u}_- \in i\RR$.  So we may this time again choose to fix
$u_-=0$ and enlarge the minimal domain by the set spanned by $\{w
\psi^m_l\}_{m=-l}^l$.

Finally, consider the case when $\tilde \beta_l = ia$.  This time we
get the equation
\begin{equation*}
  [u_l,v_l]_A = k(u_+\bar{v}_+ - u_-\bar{v}_-).
\end{equation*}
Setting $[u_l,u_l]=0$ we arrive at the condition $|u_+| = |u_-|$.  We
can choose $u_+=1$ and $u_- = -1$ to get a self-adjoint condition by
enlarging the minimal domain by the set spanned by $\{w \cos (a \log
\rho) \psi^m_l\}_{m=-l}^l$.

In order to get back to the corresponding choice of self-adjoint
extension of $\Hk - \lambda$, we multiply each basis element by
$\rho^{-1/2}$.  Since each of these basis functions is in $\rho^\mu
L^2(\TT)$ for all $\mu<1$, overall we find that $\mathcal{D}(H)
\subset \rho^\mu L^2(\TT)$ for all $\mu<1$.  This completes the proof
of Theorem \ref{theorem1}.
\end{proof}

\subsection{Some corollaries}
We now prove some consequences of Theorem \ref{theorem1.gen}. First,
its proof implies the following stronger result:

\begin{corollary}
If $\eta_0 > 0$, then $\maD(H_{\bf k}) \subset H^2(\TT) \cap
\rho^\epsilon C^0(\TT)$, for some $\epsilon>0$.
\end{corollary}

We now deduce Corollary \ref{corollary2} from Theorem \ref{theorem1}
and its proof. This corollary is specific to the periodic case.

\begin{proof} (of Corollary \ref{corollary2}).
The extension of $\Hk$ is self-adjoint, hence has real spectrum, so we
get that for $\lambda \notin \RR$, the operator
\begin{equation*}
 \lambda - \Hk : \maD(\Hk) \rightarrow L^2(\TT)
\end{equation*}
is a bounded invertible operator with bounded inverse $Q_{\vt
  k,\lambda}$, which is the resolvent of $\Hk$ at $\lambda$.  By
Theorem \ref{theorem1}, by the definition of the weighted Sobolev
spaces $\maK_a^m(\TmS)$, and by the b-Rellich lemma
\cite{AIN,meaps}, we have for some $\epsilon > 0$ that
\begin{equation*}
 \maD(\Hk) \subset \maK^2_{\epsilon}(\TmS ) \Subset L^2(\TT).
\end{equation*}
(Recall that $\Subset$ means ``compactly embedded''.)  Thus for
$\lambda \notin \RR$, the resolvent of $\Hk$,
\begin{equation*}
  \mathcal{R}_\lambda(\Hk): L^2(\TT) \rightarrow L^2(\TT)
\end{equation*}
is a compact operator.  By standard results of functional analysis, if
a self-adjoint operator has compact resolvent, then $L^2$ has a
complete orthonormal basis consisting of eigenfunctions for this
operator.
\end{proof}

\subsection{Singular functions expansion}
To prove Theorem \ref{theorem2}, we again use results of the
b-calculus, this time primarily from \cite{meaps}. Again, we prove a
more general statement that does not require the Assumption 2. Let
$\nu_0$ be as in Equation \eqref{eq.nu_0}.

\begin{theorem}\label{theorem2.gen}
Assume $\maS=\{p\}$ and $Z:= \rho^2V$ satisfies Assumption 1.  Assume
$\Hk u = \lambda u$ where $u \in \maD(\Hk)$. Then for any $m\in \ZZ_+$
and any $\nu < \nu_0$,
\begin{equation*}
    u \in    \maK^{m}_{\nu}(\TmS ).
\end{equation*}
Further, near each $p \in \maS$, where $Z(p) \neq (l+1/2)^2$ for any
$l \geq 0$, $u$ has a complete (though not unique) expansion of the
form:
\begin{equation}\label{exp}
      u = u_0 + \sum \rho^{\gamma} g_\gamma, \quad \gamma \in
      \mathcal{I}_{Z(p)},\ \Re(\gamma) > - 1/2,
\end{equation}
where the formula for $\mathcal{I}_{Z(p)}$ is given in Equation
\eqref{index} below, $u_0$ is smooth up to $\rho=0$ in polar
coordinates and vanishes to all orders there, hence is in fact smooth
on $\TT$ and vanishes to all orders at $p$, and the coefficient
functions $g_\gamma$ are smooth functions on $S^2$.  Under the
additional Assumption 2, when $Z(p)>-1/4$, the first coefficient
$g_{\eta - 1/2}$ is a constant function.
\end{theorem}

\begin{proof}
Any eigenfunction $u$ of $\Hk$ must be in its domain, thus in
$\maK^2_\nu(\TmS)$ for all $\nu< \nu_0$.  Our first goal is to improve
the degree of smoothness from $2$ to $m$ for $m \in \mathbb{N}$.  To
do this, we use the fact that any $\lambda \in \CC$, the operator
$H_{\vt{k}} -\lambda$ is Fredholm as a map between weighted Sobolev
spaces:
\begin{equation}\label{Fredholm}
  \Hk - \lambda: \maK^m_a(\TT \smallsetminus\maS) \rightarrow
  \maK^{m-2}_{a-2}(\TT \smallsetminus\maS)
\end{equation}
for all $a \in \RR$ such that $a \notin \cup_{l \in \ZZ_{\geq 0 }}
\{\beta_{l,p} + \frac{3}{2}, \alpha_{l,p}+\frac{3}{2}\}=
\mbox{Spec}_b(H_{\vt{k}}) + 3/2$.  By general b-calculus theory, the
set $\mbox{Spec}_b(H_{\vt{k}})$ is a discrete subset of $\CC$ and
furthermore, for any $\gamma_0$ and $\eta$, it has only a finite
number of elements in the strip $\gamma_0 \leq \Re(z) \leq \eta$.
Thus for any $\nu_0$, there exist arbitrarily close $\nu<\nu_0$ such
that the condition on $a$ is satisfied for $a=\nu + 2s$, where $s \in
\mathbb{N}$.  Together with standard bootstrapping arguments, this
allow us to improve the regularity of eigenfunctions of $\Hk$ in terms
of weighted Sobolev spaces to $\maK^m_{\nu}(\TmS)$ for all $m$ and
$\nu < \nu_0$.

Next, to obtain the expansion in \ref{theorem2} we use a general
result in the b-calculus literature, see {\em e.g.} \cite{meaps}, that
implies that any $u \in \cup_{m,a}\maK^m_a(\TT \smallsetminus\maS)$
which is an eigenfunction for $\Hk$ in some weighted $L^2(\TT)$ in
fact has much stronger regularity: it polyhomogeneous in $\rho$ near
each $p \in \maS$ with index set $\mathcal{I}_{Z(p)}$.  If $Z(p)
\notin \{-(l+\frac{1}{2})^2\}_{l=0}^\infty$ (for instance if
$Z(p)>-1/4$), then the index set is simply a set of complex numbers
that is finite in any strip $\gamma_0 \leq \Re(z) \leq \eta$:
\begin{equation}\label{index}
  \mathcal{I}_{Z(p)} = \bigcup_{n=0}^\infty \{\beta_{l,p} + n,
  \alpha_{l,p}+n \}_{l \in \ZZ_{\geq 0 }}.
\end{equation}
This means that around each $p \in \maS$, there exist smooth
coefficient functions $g_{\gamma} \in \maC^{\infty}(S^2)$ such that
for all $N$,
\begin{equation}\label{eq.asympt.exp}
	u_N := u - \sum \rho^{\gamma} g_\gamma \in \rho^N
        \maC^N(\overline{\TmS}), \quad \gamma \in \mathcal{I}_{Z(p)}
        \Re(\gamma) \leq N
\end{equation}
that is, $u_N$ is $N$ times continuously differentiable up to $\rho
=0$ in polar coordinates at $p$ and in these coordinates, vanishes
with all of its derivatives up to order $\rho^N$ there.  We take the
limit in the topology of the smallest $\maK^m_a$ to which $u$ belongs.

If $Z(p) \neq -(l+\frac{1}{2})^2$ for any $l \geq 0$ and if $u \in
L^2(\TT)$, when we let $N \rightarrow \infty$, we find that
\begin{equation*}
  u = u_0 + \sum \rho^{\gamma} g_\gamma, \quad
  \gamma \in \mathcal{I},  \Re(\gamma) > -1/2,
\end{equation*}
where $u_0$ is smooth up to $\rho=0$ in polar coordinates and vanishes
to all orders there, hence is in fact smooth on $\TT$ and vanishes to
all orders at $p$.

Since the set of $\gamma$ that appear in this expansion is discrete in
$\RR$, we get that the smallest exponent that appears will in fact be
somewhat better than $-3/2$.  This first exponent will be $\nu_0$ if
$Z(p) \leq 3/4$.  If $Z(p)>3/4$, then eigenfunctions will in fact be
in a space with higher weight than $\maK^2_2(\TmS)$: the weight will
be $1+\sqrt{Z(p) + 1/4}$.

Finally, we can note that the terms of the expansion of an
eigenfunction for $\Hk$ that are not in $\maK^2_2(\TmS)$ will be of
the forms determined in the proof of Theorem \ref{theorem1} (see, eg
\cite{HHM} for a proof).  So, for instance, if $Z(p)\geq -1/4$, the
leading term of any eigenfunction $u$ will be constant in $S^2$, and
$u$ minus its leading term will vanish at $p$. Further, if $Z(p)\geq
-1/4$, then the exponents $\gamma$ will all be real numbers.  Thus we
obtain:
\begin{equation}\label{exp2}
  u = u_0 + \rho^{\eta - \frac{1}{2}} g_{\eta - 1/2} + \sum
  \rho^{\gamma} g_\gamma, \quad \gamma \in \mathcal{I}_{Z(p)} \gamma >
  \eta - 1/2,
  \end{equation}
where $g_{\eta - 1/2}$ a constant.  This completes the proof of
Theorem \ref{theorem2}.
\end{proof}

\section{Invertibility\label{sec2}}

In this section we prove the boundedness and invertibility result in
Theorem \ref{theorem1.5}. From now on, we require both Assumptions 1
and 2 to be satisfied by our potential $V$.

\subsection{Preliminary results}
We begin with a few standard results lemma.

\begin{lemma}\label{lemma.Q2}
Let $m, a \in \RR$. Then
\begin{enumerate}[(i)]
\item For any $f \in \maC^{\infty}(\overline{\TmS})$, the
  multiplication map
\begin{equation*}
  \maK^m_a(\TmS) \ni u \to f u \in
  \maK^m_a(\TmS) = \rho^{a-3/2} H^m_b(\TmS)
\end{equation*}
is continuous for all $m \in \ZZ_+$ and all $a \in \RR$.
\item The operator $\Hk -\lambda$ maps $\maK_{a+1}^{m+1} (\TmS )$
to $\maK_{a-1}^{m-1} (\TmS )$ continuously.
\item The operator $\rho^{-1}B_{\vt k, \lambda}$ maps
  $\maK_{a+1}^{m+1} (\TmS )$ to $\maK_{a}^{m}
  (\TmS )$ continuously.
\item $\rho^{-1}B_{\vt k, \lambda} : \maK_{a+1}^{m+1} (\TmS) \to
  \maK_{a-1}^{m-1} (\TmS)$ is compact.
\end{enumerate}
\end{lemma}

\begin{proof}
The simple proofs of these results are the same as that of the
analogous results in \cite{HunsickerNistorSofo}, and follow directly
from properties of b-operators \cite{AIN, meaps, Lesch}.
\end{proof}
We also need the following standard lemma (again, see
\cite{HunsickerNistorSofo} for its proof).
\begin{lemma} \label{lemma.Green}
Let $a \in \RR$ be arbitrary and assume that $u \in
\maK_{1+a}^2(\TmS)$ and that $v \in \maK_{1-a}^2(\TmS)$. Then $(\Delta
u , v) + (\nabla u, \nabla v) = 0$.
\end{lemma}

We shall also need the following consequence of the general
properties of the b-calculus \cite{meaps, Schulze98}.

\begin{proposition}\label{prop.Fredholm}
Let us fix $\lambda \in \CC$ and $a \notin \{\tilde \beta_{l,p} ,
\tilde \alpha_{l,p}\} = \cup_{l \in \ZZ_{\geq 0 }} \{\beta_{l,p} +
\frac{1}{2}, \alpha_{l,p}+\frac{1}{2}\}$. Let $N$ be the number of
elements in the set $\{\tilde \beta_{l,p} , \tilde \alpha_{l,p}\}$
that are between $0$ and $a$, counted with multiplicity.  Then the
operator $H_{\vt{k}} -\lambda$ is Fredholm as a map between weighted
Sobolev spaces:
\begin{equation*}
  \Hk - \lambda: \maK^{m+1}_{a+1}(\TT \smallsetminus\maS) \rightarrow
  \maK^{m-1}_{a-1}(\TT \smallsetminus\maS)
\end{equation*}
and has index $-N$ if $a > 0$, respectively $N$ if $a < 0$.
\end{proposition}

\begin{proof}
We consider again the operator $P_{0, 0} = \rho (\Hk - \lambda) \rho$,
which is a b-differential operator. It is unitarily equivalent to
$\rho^{1/2} P_{0, 0} \rho^{-1/2}$ acting on b-Sobolev spaces (see the
proof of Theorem \ref{theorem1.gen}), which has $\{\tilde \beta_{l,p}
, \tilde \alpha_{l,p}\}$ as a b-spectrum. The result then follows from
the characterization of Fredholm b-differential operators \cite{meaps,
  KMR, LMN1}.

It remains to determine the index of $\Hk - \lambda$. Let $m_a$ be the
index of the operator for a fixed value of $a$. Then it is a standard
result that $m_a - m_b$ is given by the number of singular functions
with exponent between $a$ and $b$ \cite{KMR, meaps, Schulze98,
  MoroianuNistor}. This is enough to complete the proof.
\end{proof}

See \cite{HLNU3} for an extension of this result and for more details.

Now recall the Hardy inequality, which states that
\begin{equation}\label{eq.Hardy}
  \int_{\RR^N} r^{-2} |u(x)|^2 dx \le (2/(N-2))^2 \int_{\RR^N} |\nabla
  u(x)|^2 dx
\end{equation}
for any $u \in H^1(\RR^N)$, $N \ge 3$, where $r$ is the distance to
the origin \cite{GGM}. We can use this to prove the following
important lemma. To simplify notation, after the lemma statement, we
shall let $(u, v) := (u, v)_{L^2(\TT)}$.

\begin{lemma}\label{lemma.bounded}
There are constants $C, \gamma > 0$ such that for any $u \in
\maK_1^1(\TmS)$,
\begin{equation*}
  (\Hk u, u)_{L^2(\TT)} + C (u, u)_{L^2(\TT)} \ge \gamma (u,
  u)_{\maK_1^1(\TmS)} := \gamma \int_{\TT} \big (\rho^{-2} |u(x)|^2 +
  |\nabla u(x)|^2 \big ) dx.
\end{equation*}
\end{lemma}

\begin{proof} 
For an operator $T: \maK_1^1(\TmS) \rightarrow \maK_{-1}^{-1}(\TmS)$,
we shall write $T \ge 0$ if $(Tu, u) \ge 0$ for all $u \in
\maK_{1}^1(\TmS)$.  Now let $\phi \ge 0$ be a smooth function
$\mathbb{T}$ that is equal to $1$ in a small neighborhood of $\maS$
and has support on the set where $\rho(x)$ is given by the distance to
$\maS$ and let $V_0 (x) = Z(p)\phi(x) \rho^{-2}(x)$ for $x$ in the
support of $\phi$ and close to $p \in \maS$. Outside the support of
$\phi$, we let $V_0 = 0$.  Then Hardy's inequality applied to
$\phi^{1/2} u$, which we can think of as living now on $\mathbb{R}^3$
rather than $\mathbb{T}$, gives
\begin{equation}\label{eq.Hardy.local}
  \big( \phi^{1/2} (-\Delta + z V_0)\phi^{1/2} u, u ) \ge 0 \quad
  \mbox{and } \big( \phi^{1/2} (-\Delta)\phi^{1/2} u, u ) \ge 0.
\end{equation}
We can think of this as saying that the most singular part of the
operator $\Hk$, that is, $T=\phi^{1/2} (-\Delta + z V_0)\phi^{1/2}$,
satisfies $T\geq 0$.  We will prove Lemma \ref{lemma.bounded} by
decomposing the operator $\Hk+C$ as a sum of four operators
\begin{equation*}
  \Hk = T_1+ T_2 + T_{3,C} + T_{4,C},
\end{equation*}
which we will show are all bounded from below for sufficiently large
$C$.

Recall we can write
\begin{equation*}
       \Hk = -\Delta + V_0 + V_1 + \rho^{-1}B_{\rm k. 0},
\end{equation*}
where $V_1 := V-V_0$ satisfies $\rho V_1 \in C(\mathbb{T})$ and
$\rho^{-1}B_{\rm k. 0}$ is a first order differential operator over
$\mathbb{T}$ with smooth coefficients.

Assumption 2 and Equation \eqref{eq.Hardy.local} imply that for
$\epsilon<1$, the operator
\begin{equation}\label{op.positive}
  T_1 := (1-\epsilon)T \ge 0.
\end{equation}
Fix any suitable value for $\epsilon>0$.  Then we can write $\Hk +C$
in terms of $T_1$ by decomposing in terms of the multiplication
operators $\phi^{1/2}$ and $(1-\phi^{1/2}$:
\begin{equation}
  \Hk + C = \epsilon T + T_1 -\psi^{1/2}\Delta \psi^{1/2} -
  (1-\phi^{1/2})\Delta(1-\phi^{1/2}) + V_1 +R_1,
 \end{equation}
where $R_1$ is a first order differential operator with smooth
coefficients and $\psi = 2\phi^{1/2}(1-\phi^{1/2})$.

Let $T_2 := -\psi^{1/2}\Delta \psi^{1/2} - (1-\phi^{1/2})
\Delta(1-\phi^{1/2})$.  Then $T_2 \geq 0$ by the Hardy equality
applied to $\psi^{1/2} u$ and to $(1-\phi^{1/2})u$.  Define $T_3 :=
-\epsilon \Delta + R_1 + C/2$ and $T_4 = \epsilon \phi\rho^{-2} + V_1
+ C/2$.  We claim that for $C$ large enough, $T_3 \ge 0$ and $T_4 \ge
0$, which will prove the result.

The proof that $T_4 \geq 0$ for $C>>0$ follows from a straightforward
calculation minimizing the function $\epsilon \phi\rho^{-2} + V_1$.
The proof that $T_3 \geq 0$ for $C>>0$ is basically the same as the
proof that a Schr\"odinger operator with periodic Coulumb type
potential is bounded below.  This is proved, for example, in
\cite{HunsickerNistorSofo}.
\end{proof}

Note that the above lemma implies that $\Hk$ is bounded from below as
an operator $\maK_1^1(\TmS)\rightarrow \maK_{-1}^{-1}(\TmS)$, which is
the special case of Theorem \ref{theorem1.5} when $m=a=0$.  In
addition, if we define the form $\alpha(u,v):=((\Hk+C)u,v),$ where the
right-hand side is the natural pairing between elements of
$\maK^{-1}_{-1}$ and $\maK^1_1$, then this lemma implies that
$\alpha(u,v)$ satisfies the assumptions of the Lax-Milgram lemma for
the vector space $V=\maK_1^1(\TmS)$. This and C\'{e}a's lemma imply
that for any element $u \in \maK_1^1(\TmS)$ and any finite dimensional
subspace $V \subset \maK_1^1(\TmS)$ we can construct a unique
(Galerkin) approximation $u_V \in V$ for $u$ that, up to a multiple
independent of $u$, is the best approximation for $u$ in $V$.

If we could also use the $\maK_1^1$ norm in our approximation results,
we would now have the necessary tools to prove it.  However, we need
to use the slightly smaller space $\maK_{a+1}^1$ instead.  Thus we the
stronger result, Theorem \ref{theorem1.5} to ensure the Lax-Milgram
theorem and C\'ea's lemma apply to the analogous form on these spaces.

We shall also need the following regularity result.

\begin{proposition}\label{prop.regularity}
Let $a, \lambda \in \RR$, $m \in \ZZ_{+}$. There exists a constant $C
> 0$ such if $u \in \maK_{a+1}^{1}(\TmS)$ and $(\lambda + \Hk)u \in
\maK_{a-1}^{m-1}(\TmS)$ then $u \in \maK_{a+1}^{m+1}(\TmS)$ and
\begin{equation*}
  \|u\|_{\maK_{a+1}^{m+1}} \le C \big ( \|(\lambda + \Hk)
  u\|_{\maK_{a+1}^{m+1}} + \|u\|_{\maK_{a+1}^{1}} \big ).
\end{equation*}
\end{proposition}

\begin{proof}
We consider again the operator $P_{0, 0} = \rho (\Hk - \lambda) \rho$,
which is a b-differential operator. Our result then follows from the
regularity for b-pseudodifferential operators \cite{meaps, AIN}.
\end{proof}

We now complete the proof of Theorem
\ref{theorem1.5} as follows.

\begin{proof} (of Theorem \ref{theorem1.5}).
As in \cite{HMN}, by regularity for b-differential operators, if we
prove our result for $m = 0$, then the regularity result of
Proposition \ref{prop.regularity} implies it for all $m \ge 0$. We shall
thus assume $m=0$ and focus on extending Lemma \ref{lemma.bounded},
where $a=0$, to the case when $|a| < \eta$.

Fix $C$ as in Lemma \ref{lemma.bounded}.  Let $D_a := C + \Hk :
\maK^{1}_{a+1}(\TmS ) \to \maK^{-1}_{a-1}(\TmS),$ that is to say, $C +
\Hk$ with fixed domain and range. As usual, we may identify the dual
of $\maK_{b}^{1}(\TmS)$ with the space $\maK_{-b}^{-1}(\TmS)$ using
the $L^2$-inner product. Then using the symmetry of $\Hk$, we find
that $D_a^* =D_{-a}$.

By Lemma \ref{lemma.bounded}, the operator $D_0$ is invertible.  By
basic results in b-calculus, $D_a$ is Fredholm for $|a| < \eta$ since
the weighted spaces in its domain and range do not correspond to an
indicial root as calculated in the previous section (see Proposition
\ref{prop.Fredholm}).  Hence, for such $a$, the family $\rho^a D_a
\rho^{-a}$ is a continuous family of Fredholm operators between the
same pair of spaces. Since index is constant over such families, we
have that $\ind(D_a) = 0$ for all $0\leq a < \eta$. We want to know
these operators are all isomorphisms. By the index calculation, it now
suffices to show they are all injective.

The inclusion $\maK^{1}_{a+1}(\TmS ) \subset \maK^{1}_{1}(\TmS)$
allows us to compute $(D_a u , u) = (\nabla u, \nabla u) + (u, u)$ for
$u \in \maK^{1}_{a+1}(\TmS)$, by Lemma \ref{lemma.Green}.  Assume $D_a
u = 0$, then $(D_a u, u) = 0$, and hence $u = 0$. This implies that
the operator $D_a$ is injective for $0 \le a < \eta$.  Since it is
Fredholm of index zero, it is also an isomorphism.  This proves our
result for $0 \le a < \eta$.  To prove the result for $-\eta < a \le
0$, we take adjoints and use $D_a = (D_{-a})^*$.

By the characterization in \cite{GM} of the Friedrichs extension of a
b-operator which is bounded below, we can see that the extension we
constructed in Theorem \ref{theorem1} is in fact the Friedrichs
extension of $\Hk$.  The proof of Theorem \ref{theorem1.5} is now
complete.
\end{proof}

The fact that the domain of the Friedrichs extension is $(C -
\Hk)^{-1}(L^2(\TmS))$ and the theorem we have just proved give us a
second way to identify the domain of the Friedrichs extension of
$\Hk$.  Following the method of \cite{HunsickerNistorSofo}, we see
that when $\Hk$ is Fredholm on $\maK_{2}^{2}(\TmS)$, the domain of the
Friedrichs extension of $\Hk$ consists of the span of
$\maK_{2}^{2}(\TmS)$ and of the singular functions that are in
$\maK_{1}^{1}(\TmS)$ but are not in $\maK_{2}^{2}(\TmS)$. This can be
used to obtain an alternative proof of Theorem \ref{theorem1} if $V$
satisfies both Assumptions 1 and 2, as follows.

\begin{proposition} Let $C_0$ be as in Theorem \ref{theorem1.5} and
$W_s$ be as in Equation \eqref{eq.def.Ws}. Assume the set $\{\tilde
  \beta_{l,p} , \tilde \alpha_{l,p}\}$ does not contain 1. Then for
  $\lambda > C_0$, the map
\begin{equation*}
  \lambda + \Hk : \maK_{2}^{m+1}(\TmS) + W_s \to \maK_{0}^{m-1}(\TmS)
\end{equation*}
is an isomorphism.
\end{proposition}

\begin{proof}
Let $T := \lambda + \Hk$ with the indicated domain and codomain.
Proposition \ref{prop.Fredholm} shows that $T$ is Fredholm with index
zero. Since
\begin{equation*}
  \maK_{2}^{m+1}(\TmS) + W_s \subset \maK_{1}^{1}(\TmS),
\end{equation*}
Theorem \ref{theorem1.5} shows that $T$ is injective. Hence it is also
surjective, hence an isomorphism.
\end{proof}

\section{Extensions and numerical tests\label{sec.last}}

We now discuss the extension to the non-compact case and indicate
some applications to numerical methods.

\subsection{The non-compact case\label{ssec.new}}

Most of our results in the previous sections extend to the non-compact
case. Let $\overline{\RR^3}_{rad}$ be the radial compactification of
$\RR^3$. We assume that the set of singular points $\maS \subset
\RR^3$ is finite and we replace each of the points in $\maS$ with a
two-sphere (that is, we blow up the singular points). Let $\RmS$
denote the resulting compact manifold with boundary. By $\rho$ we
denote a continuous function $\rho : \RR^3 \to [0, 1]$ that is smooth
outside $\maS$, close to each $p \in \maS$ it has the form $\rho(x) =
|x - p|$, and it is constant equal to $1$ outside a compact set. (Thus
the difference with the function $\rho$ considered in the periodic
case is that now $\rho$ is constant equal to $1$ in a neighborhood of
infinity.) Then in the non-compact case, our Assumption 1 on $Z :=
\rho^2 V$ is replaced with
\begin{equation}\label{eq.def.Zprime}
  \text{\bf Assumption }1^\prime: \qquad\ Z := \rho^2 V \in
  \maC^\infty(\RmS) \cap \maC(\RR^3).
\end{equation}
Assumption 2 remains unchanged.

We consider now $H = -\Delta + V$ instead of the restrictions $\Hk$.
Assumptions $1^\prime$ and 2 allow us to extend to $H$ all the results
for $\Hk$ of the previous sections, except Corollary \ref{corollary2}
and Proposition \ref{prop.Fredholm}. The weighted Sobolev spaces
$\maK_{a}^{m}(\RmS)$ are defined in the same way (but using the new
function $\rho$).

Let $b_e$ the infimum of $V$ on the sphere at infinity. Then Corollary
\ref{corollary2} must be replaced with the following characterization
of the essential spectrum $\sigma_e(H)$ of $H$:
\begin{equation}
  \sigma_e(H) = [b_e, \infty).
\end{equation}
To prove this result, one needs also the Fredholm conditions for
operators in the scattering or SG calculus \cite{LMN1,
  MelroseScattering, SchroheSI, Parenti}. Then in Proposition
\ref{prop.Fredholm} one has to take $\lambda < b_e$. Of course, in
Theorem \ref{theorem1.5} one will have $C_0 > - b_e$.

However, in the non-compact case, for applications to numerical
methods, our results on eigenvalues and eigenfunctions must be
complemented by decay properties at infinity.  The following is proved
as in \cite{AmmannCarvalhoNistor}, Theorem 4.4. See also
\cite{AgmonDecay, PuriceMagnetic, PuriceDecay}. Let $r: \RR^3 \to \RR$
be a smooth function such $r(x) = |x|$ for $x$ outside a compact set.

\begin{theorem}\label{thm.eigen.decay}  Let $V$ be a potential
satisfying Assumptions $1^\prime$ and 2.  Also, let $0 < \epsilon <
V(x) - \lambda $ for $x$ outside a compact set and be $u$ an
eigenvector of $\Hk$ corresponding to $\lambda$. Then $e^{\epsilon r}
u \in \maK_{\nu}^{m}(\RmS)$.
\end{theorem}

Under Assumptions $1^\prime$ and 2, a perturbation argument further
yields following result on the decay properties of the eigenfunctions
and the solutions of the equation $(C + \Hk)u = f$.

\begin{theorem} Let us assume that $C + \Hk : \maK_{1}^{1}(\RmS) \to
\maK_{-1}^{-1}(\RmS)$ is invertible (which is the case if $C > C_0$,
with $C_0$ as in Theorem \ref{theorem1.5}), then for $|a|$ and
$|\epsilon|$ small
\begin{equation*}
  C + \Hk : e^{\epsilon r} \maK_{a+1}^{m+1}(\RmS) \to e^{\epsilon r}
  \maK_{a-1}^{m-1}(\RmS)
\end{equation*}
is again invertible.
\end{theorem}

\begin{proof} The proof uses the same arguments used in the proof
of Theorem \ref{theorem1.5}, the continuity of the family
$e^{\epsilon r} \Hk e^{- \epsilon r}$ in $\epsilon$, and the regularity
result \ref{prop.regularity}.
\end{proof}

\subsection{Applications and numerical tests\label{ssec3}}

Let $u$ be an eigenvector of $\Hk$ or the solutions of equations of
the form $(\lambda + \Hk)u = f$, with $f$ smooth enough. Our results
give smoothness properties for $u$. They also give decay properties of
$u$ in the non-periodic case. These properties, in turn, can be used
to obtain approximation properties of $u$. Standard numerical methods
results (C\'{e}a's lemma or the results reviewed in \cite{BabuOsborn})
then lead to error estimates in the Finite Element Method for the
numerical solutions of the equation $(C + \Hk)u = f$ or for the
eigenfunctions of $\Hk$. We have tested these approximation results in
the periodic case using, first, a graded mesh and, second, augmented
plane waves. In both cases, the tests are in good agreement with our
theoretical results. These numerical and the needed approximation
results will be discussed in full detail in the second and fourth
parts of our paper \cite{HLNU2, HLNU4}.

%\bibliographystyle{plain}
%\bibliography{num}

\begin{thebibliography}{10}

\bibitem{AgmonDecay}
S.~Agmon.
\newblock {\em Lectures on exponential decay of solutions of second-order
  elliptic equations: bounds on eigenfunctions of {$N$}-body {S}chr\"odinger
  operators}, volume~29 of {\em Mathematical Notes}.
\newblock Princeton University Press, Princeton, NJ, 1982.

\bibitem{AmmannCarvalhoNistor}
B.~Ammann, C.~Carvalho, and V.~Nistor.
\newblock {R}egularity for eigenfunctions of {S}chr\"{o}dinger operators.
\newblock Lett. Math. Phys. 2012.

\bibitem{AIN}
B.~Ammann, A.~Ionescu, and V.~Nistor.
\newblock Sobolev spaces on {L}ie manifolds and regularity for polyhedral
  domains.
\newblock {\em Documenta Math (electronic)}, 11:161--206, 2006.

\bibitem{aln2}
B.~Ammann, R.~Lauter, and V.~Nistor.
\newblock Pseudodifferential operators on manifolds with a {L}ie structure at
  infinity.
\newblock {\em Ann. of Math. (2)}, 165(3):717--747, 2007.

\bibitem{Bespalov}
D.~Arroyo, A.~Bespalov, and N.~Heuer.
\newblock On the finite element method for elliptic problems with degenerate
  and singular coefficients.
\newblock {\em Math. Comp.}, 76(258):509--537 (electronic), 2007.

\bibitem{BabuOsborn}
I.~Babu{\v{s}}ka and J.~Osborn.
\newblock Eigenvalue problems.
\newblock In {\em Handbook of numerical analysis, {V}ol.\ {II}}, Handb. Numer.
  Anal., II, pages 641--787. North-Holland, Amsterdam, 1991.

\bibitem{BNZ3D2}
C.~Bacuta, V.~Nistor, and L.~Zikatanov.
\newblock Improving the rate of convergence of high-order finite elements on
  polyhedra. {II}. {M}esh refinements and interpolation.
\newblock {\em Numer. Funct. Anal. Optim.}, 28(7-8):775--824, 2007.

\bibitem{Westphal}
S.~Bidwell, M.~Hassell, and C.R. Westphal.
\newblock A weighted least squares finite element method for elliptic problems
  with degenerate and singular coefficients.
\newblock Math. Comp. (in press).

\bibitem{Dauge}
M.~Dauge.
\newblock {\em Elliptic boundary value problems on corner domains}, volume 1341
  of {\em Lecture Notes in Mathematics}.
\newblock Springer-Verlag, Berlin, 1988.
\newblock Smoothness and asymptotics of solutions.

\bibitem{ferrero}
V.~Felli, A.~Ferrero, and S.~Terracini.
\newblock Asymptotic behavior of solutions to {S}chr\"{o}dinger equations near
  an isolated singularity of the electromagnetic potential.
\newblock preprint.

\bibitem{Felli2}
V.~Felli, A.~Ferrero, and S.~Terracini.
\newblock Asymptotic behavior of solutions to {S}chr\"odinger equations near an
  isolated singularity of the electromagnetic potential.
\newblock {\em J. Eur. Math. Soc. (JEMS)}, 13(1):119--174, 2011.

\bibitem{Felli1}
V.~Felli, E.~Marchini, and S.~Terracini.
\newblock On the behavior of solutions to {S}chr\"odinger equations with dipole
  type potentials near the singularity.
\newblock {\em Discrete Contin. Dyn. Syst.}, 21(1):91--119, 2008.

\bibitem{Flad2}
H.-J. Flad, W.~Hackbusch, and R.~Schneider.
\newblock Best {$N$}-term approximation in electronic structure calculations.
  {II}. {J}astrow factors.
\newblock {\em M2AN Math. Model. Numer. Anal.}, 41(2):261--279, 2007.

\bibitem{Flad3}
H.-J. Flad and Harutyunyan.
\newblock Ellipticity of quantum mechanical hamiltonians in the edge algebra.
\newblock preprint ArXiv:1103.0207v1.

\bibitem{Flad1}
H.-J. Flad, R.~Schneider, and B.-W. Schulze.
\newblock Asymptotic regularity of solutions to {H}artree-{F}ock equations with
  {C}oulomb potential.
\newblock {\em Math. Methods Appl. Sci.}, 31(18):2172--2201, 2008.

\bibitem{Fournais1}
S.~Fournais, M.~Hoffmann-Ostenhof, T.~Hoffmann-Ostenhof, and
  T.~{\O}stergaard~S{\o}rensen.
\newblock Analytic structure of solutions to multiconfiguration equations.
\newblock {\em J. Phys. A}, 42(31):315208, 11, 2009.

\bibitem{GGM}
F.~Gazzola, H.-C. Grunau, and E.~Mitidieri.
\newblock Hardy inequalities with optimal constants and remainder terms.
\newblock {\em Trans. Amer. Math. Soc.}, 356(6):2149--2168 (electronic), 2004.

\bibitem{GKM}
J.~Gil, T.~Krainer, and G.~Mendoza.
\newblock Geometry and spectra of closed extensions of elliptic cone operators.
\newblock {\em Canad. J. Math.}, 59(4):742--794, 2007.

\bibitem{GM}
J.~Gil and G.~Mendoza.
\newblock Adjoints of elliptic cone operators.
\newblock {\em Amer. J. Math.}, 125(2):357--408, 2003.

\bibitem{Hamaekers}
M.~Griebel and J.~Hamaekers.
\newblock Tensor product multiscale many-particle spaces with finite-order
  weights for the electronic {S}chr\"odinger equation.
\newblock {\em Zeitschrift f\"ur Physikalische Chemie}, 224:527--543, 2010.

\bibitem{grieser}
D.~Grieser.
\newblock Basics of the {$b$}-calculus.
\newblock In {\em Approaches to singular analysis ({B}erlin, 1999)}, volume 125
  of {\em Oper. Theory Adv. Appl.}, pages 30--84. Birkh\"auser, Basel, 2001.

\bibitem{HHM}
T.~Hausel, E.~Hunsicker, and R.~Mazzeo.
\newblock Hodge cohomology of gravitational instantons.
\newblock {\em Duke Math. J.}, 122(3):485--548, 2004.

\bibitem{Siedentop}
B.~Helffer and H.~Siedentop.
\newblock Regularization of atomic {S}chr\"odinger operators with magnetic
  field.
\newblock {\em Math. Z.}, 218(3):427--437, 1995.

\bibitem{HLNU2}
E.~Hunsicker, H.~Li, V.~Nistor, and V.~Uski.
\newblock Analysis of {S}chr\"odinger operators with inverse square potentials
  {II}: {FEM} and approximation of eigenfunctions in the periodic case.
\newblock preprint, submitted for publication.

\bibitem{HLNU3}
E.~Hunsicker, H.~Li, V.~Nistor, and V.~Uski.
\newblock Analysis of {S}chr\"odinger operators with inverse square potentials
  {III}: the $n$ dimensional case.
\newblock {W}ork in progress, tentative title.

\bibitem{HLNU4}
E.~Hunsicker, H.~Li, V.~Nistor, and V.~Uski.
\newblock Analysis of {S}chr\"odinger operators with inverse square potentials
  {IV}: plane waves and orbital functions.
\newblock {W}ork in progress, tentative title.

\bibitem{HunsickerNistorSofo}
E.~Hunsicker, V.~Nistor, and J.~Sofo.
\newblock Analysis of periodic {S}chr\"odinger operators: regularity and
  approximation of eigenfunctions.
\newblock {\em J. Math. Phys.}, 49(8):083501, 21, 2008.

\bibitem{PuriceMagnetic}
V.~Iftimie, M.~M{\u{a}}ntoiu, and R.~Purice.
\newblock Magnetic pseudodifferential operators.
\newblock {\em Publ. Res. Inst. Math. Sci.}, 43(3):585--623, 2007.

\bibitem{PuriceDecay}
V.~Iftimie and V.~Purice.
\newblock Eigenfunctions decay for magnetic pseudodifferential operators.
\newblock preprint.

\bibitem{Kato51}
T.~Kato.
\newblock Fundamental properties of {H}amiltonian operators of {S}chr\"odinger
  type.
\newblock {\em Trans. Amer. Math. Soc.}, 70:195--211, 1951.

\bibitem{Kato57}
T.~Kato.
\newblock On the eigenfunctions of many-particle systems in quantum mechanics.
\newblock {\em Comm. Pure Appl. Math.}, 10:151--177, 1957.

\bibitem{kondratiev67}
V.~A. Kondrat{\cprime}ev.
\newblock Boundary value problems for elliptic equations in domains with
  conical or angular points.
\newblock {\em Transl. Moscow Math. Soc.}, 16:227--313, 1967.

\bibitem{KMR}
V.~Kozlov, V.~Maz{\cprime}ya, and J.~Rossmann.
\newblock {\em Spectral problems associated with corner singularities of
  solutions to elliptic equations}, volume~85 of {\em Mathematical Surveys and
  Monographs}.
\newblock American Mathematical Society, Providence, RI, 2001.

\bibitem{LMN1}
R.~Lauter, B.~Monthubert, and V.~Nistor.
\newblock Pseudodifferential analysis on continuous family groupoids.
\newblock {\em Doc.\ Math.}, 5:625--655 (electronic), 2000.

\bibitem{Joerg}
R.~Lauter and J.~Seiler.
\newblock Pseudodifferential analysis on manifolds with boundary---a comparison
  of {$b$}-calculus and cone algebra.
\newblock In {\em Approaches to singular analysis ({B}erlin, 1999)}, volume 125
  of {\em Oper. Theory Adv. Appl.}, pages 131--166. Birkh\"auser, Basel, 2001.

\bibitem{Lesch}
M.~Lesch.
\newblock {\em Operators of {F}uchs type, conical singularities, and asymptotic
  methods}, volume 136 of {\em Teubner-Texte zur Mathematik [Teubner Texts in
  Mathematics]}.
\newblock B. G. Teubner Verlagsgesellschaft mbH, Stuttgart, 1997.

\bibitem{HMN}
H.~Li, A.~Mazzucato, and V.~Nistor.
\newblock Analysis of the finite element method for transmission/mixed boundary
  value problems on general polygonal domains.
\newblock {\em Electron. Trans. Numer. Anal.}, 37:41--69, 2010.

\bibitem{Martin}
R.~Martin.
\newblock {\em Electronic structure: basic theory and practical methods}.
\newblock Cambridge University Press, 2004.

\bibitem{meaps}
R.B. Melrose.
\newblock {\em {The Atiyah-Patodi-Singer index theorem.}}
\newblock {Research Notes in Mathematics (Boston, Mass.). 4. Wellesley, MA: A.
  K. Peters, Ltd.. xiv, 377 p. }, 1993.

\bibitem{MelroseScattering}
R.B. Melrose.
\newblock {\em Geometric scattering theory}.
\newblock Stanford Lectures. Cambridge University Press, Cambridge, 1995.

\bibitem{MoroianuNistor}
S.~Moroianu and V.~Nistor.
\newblock Index and homology of pseudodifferential operators on manifolds with
  boundary.
\newblock In {\em Perspectives in operator algebras and mathematical physics},
  volume~8 of {\em Theta Ser. Adv. Math.}, pages 123--148. Theta, Bucharest,
  2008.

\bibitem{MorozSchmidt}
S.~Moroz and R.~Schmidt.
\newblock Nonrelativistic inverse square potential, scale anomaly, and complex
  extension.
\newblock Preprint hep-th/0909.3477v3, 2010.

\bibitem{Parenti}
C.~Parenti.
\newblock Operatori pseudodifferentiali in {$\RR^n$} e applicazioni.
\newblock {\em Annali Mat. Pura ed App.}, 93:391--406, 1972.

\bibitem{SchroheSI}
E.~Schrohe.
\newblock {Spectral invariance, ellipticity, and the Fredholm property for
  pseudodifferential operators on weighted Sobolev spaces.}
\newblock {\em Ann. Global Anal. Geom.}, 10(3):237--254, 1992.

\bibitem{Schulze98}
B.W. Schulze.
\newblock {\em {Boundary value problems and singular pseudo-differential
  operators.}}
\newblock {Wiley-Interscience Series in Pure and Applied Mathematics.
  Chichester: John Wiley \& Sons.}, 1998.

\bibitem{Schwab}
C.~Schwab and R.~Stevenson.
\newblock Space-time adaptive wavelet methods for parabolic evolution problems.
\newblock {\em Math. Comp.}, 78(267):1293--1318, 2009.

\bibitem{VasyReg}
A.~Vasy.
\newblock Propagation of singularities in many-body scattering.
\newblock {\em Ann. Sci. \'Ecole Norm. Sup. (4)}, 34(3):313--402, 2001.

\bibitem{Zuazua}
J.~Vazquez and E.~Zuazua.
\newblock The {H}ardy inequality and the asymptotic behaviour of the heat
  equation with an inverse-square potential.
\newblock {\em J. Funct. Anal.}, 173(1):103--153, 2000.

\bibitem{Sprung}
H.~Wu and D.W.L Sprung.
\newblock Inverse-square potential and the quantum votex.
\newblock {\em Physical Review A}, 49:4305--4311, 1994.

\bibitem{Yserentant3}
H.~Yserentant.
\newblock {\em Regularity and approximability of electronic wave functions},
  volume 2000 of {\em Lecture Notes in Mathematics}.
\newblock Springer-Verlag, Berlin, 2010.

\end{thebibliography}

\def\cprime{$'$} \def\ocirc#1{\ifmmode\setbox0=\hbox{$#1$}\dimen0=\ht0
  \advance\dimen0 by1pt\rlap{\hbox to\wd0{\hss\raise\dimen0
  \hbox{\hskip.2em$\scriptscriptstyle\circ$}\hss}}#1\else {\accent"17 #1}\fi}

\end{document}